\documentclass[12pt,english]{amsart}
\usepackage{amssymb,amsmath,amscd,graphicx,fontenc,bbold,amsthm,mathrsfs,mathtools}
\usepackage[pdftex,bookmarks,colorlinks,breaklinks]{hyperref} 
 \usepackage{comment}
\hypersetup{linkcolor=blue,citecolor=red,filecolor=dullmagenta,urlcolor=blue}
\newtheorem{theorem}{Theorem}[section]

\newtheorem{lemma}[theorem]{Lemma}

\theoremstyle{definition}
\newtheorem{remark}{Remark}

\usepackage[top=3cm,bottom=3cm,right=2.3cm,left=2.3cm,twoside=false]{geometry}
\title{An Erd\H{o}s-Kac theorem for Smooth and Ultra-smooth integers}
\author{Marzieh Mehdizadeh}
\address{ D\'{e}partment de Math\'{e}matiques et Statistique,Universit\'{e} de Montr\'{e}al, CP 6128, succ.
Centre-ville, Montr\'{e}al, QC, Canada H3C 3J7.}
\email{marzieh.mehdizadeh@gmail.com}
\date{}
\begin{document}

\maketitle

\begin{abstract}

We prove an Erd\H{o}s-Kac type of theorem for the set $S(x,y)=\{n\leq x: p|n \Rightarrow p\leq y \}$. If $\omega (n)$ is the number of prime factors of $n$, we prove that the distribution of $\omega(n)$ for $n \in S(x,y)$ is Gaussian for a certain range of $y$ using method of moments. The advantage of the present approach is that it recovers classical results for the range $u=o(\log \log x )$ where $u=\frac{\log x}{\log y}$, with a much simpler proof. 

\end{abstract}

\section{Introduction} 
For an integer $n\geq 2$, let $\omega(n)$ denote the number of distinct prime divisors of $n$. In 1940, Erd\H{o}s and Kac~\cite{Erdos-Kac} in their celebrated  work  studied the distribution of $\omega(n)$ in the interval $[2,N]$. The theorem states that for any real number $x$, we have
\begin{equation} \label{erdos}
 \lim_{N \rightarrow \infty}\frac{1}{N} \#\left\{ n \le N : \frac{\omega(n) - \log \log n}{\sqrt{\log \log n}} \le x \right\}= \Phi(x).
\end{equation}
where $\Phi(x)$ is the \textit{ normal distribution function} defined by
$$\Phi(x):= \frac{1}{\sqrt{2\pi}}\int_{-\infty}^{x} e^{-\frac{t^2}{2}} dt$$
There are several proofs of Erd\H{o}s-Kac Theorem. For instance, it has been proved by Billingsley ~\cite{Bill} and Granville and Soundararajan ~\cite{Andrew} using the method of moments and sieve theory.
Different variations of this theorem have been considered by several authors. In the present note, we shall study the Erd\H{o}s-Kac theorem for $y-$smooth numbers. Recall that
$$S(x, y) := \{n \le x : P (n) \le y \} \qquad x \ge y \ge 2, $$
is the set of $y-$smooth integers, where $P(n)$ is defined as the largest prime factor of $n$, with the convention $P(1)=1$. Also, recall that we set
$$ \Psi(x, y) := | S(x, y)| \qquad x \ge y \ge 2. $$

The main goal of this result is to prove an analogue of \eqref{erdos} with the set $S(x,y)$ in the range
\begin{equation}\label{range}
 \qquad u = o(\log \log y),
\end{equation}
where, as always, $$u:=\frac{\log x}{\log y}.$$

Hildebrand~\cite{Hild}, Alladi~\cite{Alladi}, and Hensley~\cite{Hensley} have considered the distribution of prime divisors of $y-$smooth integers in different ranges of $y$.\\
Hensley proved an Erd\H{o}s-Kac type theorem when $u$ lies in the range 
$$(\log y)^{1/3} \leq u \leq \frac{\sqrt{y}}{2\log y}.$$
By using different method Alladi obtained an analogue of the Erd\H{o}s-Kac Theorem for the following range 
$$u\leq \exp(\log y)^{3/5-\epsilon}.$$ Later, Hildebrand extended previous results to include the range

$$y\geq 3 \,\,\,\,\,\,\, u \geq (\log y)^{20},$$
which is a completion of Alladi and  Hensley's results.\\

Although \eqref{range} does not cover Alladi's, Hensley's and Hildebrand's ranges, our applied method is completely different and much easier than the methods used by previous authors. \\
Our approach is based on the method of moments as Billinglsley used in~\cite{Bill}. We will introduce some approximately independent random variables, and by the Central Limit Theorem, we shall show that this random variables have a normal distribution, then by applying method of moments we get our desired result in \eqref{erdos}.\\ 
The first step of the proof is to apply a truncation on number prime factors. This idea is from original proof of Erd\H{o}s-Kac Theorem ~\cite{Erdos-Kac}.\\\\
For a given real number $y$, set
$$\phi(y): = (\log \log y)^{\sqrt{\log \log \log y}},$$
then $y^{\frac{1}{\phi(y)}}$ is a function that helps us to sieve out all primes exceeding $y^{\frac{1}{\phi(y)}}$, and we will show the contribution of sieved primes is negligible in understanding the distribution of $\omega(n)$.
Before stating the main result, we begin introducing some notation. Let $\omega (n)$ is the number of distinct prime divisors of a $y-$smooth number, namely
$$\omega(n):= \sum_{p\le y} \mathbb{1}_{p|n}(n),$$
where $\mathbb{1}_{p|n}(n)$ is $1$ and $0$ according to the prime $p$ divides $n$ or not.\\
 Let $\mu_\omega(x,y)$ be the mean value of $\omega(n)$, more formally
$$\mu_\omega(x,y) := \mathbb{E}_{n \in S(x, y)} [\omega (n)]= \frac{1}{\Psi(x, y)} \sum_{n \in S(x, y)} \omega(n),$$
and $\sigma_\omega^2(x,y)$ is the variance of $\omega(n)$, defined by
$$\sigma_\omega^2(x,y):= \mathbb{E}\left[\left(\omega(n)-\mu(x,y)\right)^2\right].$$\\
Now we are ready to state the main theorem. 
\begin{theorem}\label{Theorem 1}
For any real number $z$, we have

\begin{equation} \label{result}
\frac{1}{\Psi (x,y)} \# \{n \in S(x, y):  \frac{\omega(n) - \log \log y}{\sqrt{\log \log y}} \le z \} \rightarrow \Phi(z) \quad (y\to \infty)
\end{equation}
holds in the range \eqref{range}.
\end{theorem}
Theorem \ref{Theorem 1} is proved in Section $3$. The proof relies on the method of moments and the estimates for $\Psi(x/d,y)/\Psi(x,y)$.\\\\
Let
$$U(x,y):=\{n\leq x : p^{v}||n \Rightarrow v\leq v_p \}$$
be the set of $y-$ultra-smooth integers whose canonical decomposition is free of prime powers exceeding $y$ , where
$$v_p:= \left\lfloor\frac{\log y}{\log p} \right\rfloor.$$
We define
$$\Upsilon(x,y):=\big|U(x,y)\big|.$$
We also have the following theorem 
\begin{theorem}\label{Theorem 2}
For any real number $z$, we have

\begin{equation} \label{result}
\frac{1}{\Upsilon (x,y)} \# \{n \in U(x, y):  \frac{\omega(n) -\log \log y}{\sqrt{\log \log y}} \le z \} \rightarrow \Phi(z) \quad (y\to \infty)
\end{equation}
holds in the range \eqref{range}.

\end{theorem}
The proof of Theorem \ref{Theorem 2} relies on the method of moments and the local behaviour of the function $\Upsilon(x,y)$. By recalling  ~\cite[Corollary 1.3.]{Ultrafriable}, for $u=o(\log \log y)$, we have
$$\frac{\Upsilon(x/d,y)}{\Upsilon(x,y)}= \frac{\Psi(x/d,y)}{\Psi(x,y)}\left\lbrace 1+O\left(\frac{u\log 2u}{\sqrt{y}\log y}\right)\right\rbrace,$$
that is
$$\frac{\Upsilon(x/d,y)}{\Upsilon(x,y)}\sim\frac{\Psi(x/d,y)}{\Psi(x,y)} \quad \text{as} \quad y\to \infty.$$
Considering this relation between the local behaviour of $\Upsilon(x,y)$ and $\Psi(x,y)$ gives us a similar proof as Theorem \ref{Theorem 1}, so we shall avoid proving this theorem.

\begin{center}Acknowledgement 
\end{center}
I would like to thank Andrew Granville and Dimitris Koukoulopoulos for all their advice and encouragement as well as their valuable comments on the earlier version of the present paper. I am also grateful to Adam Harper, Sary Drappeau and Oleksiy Klurman for helpful conversations. 

\section{Preliminaries}
Here we briefly recall some standard facts from probability theory (See Feller ~\cite{Feller} for more details) and we shall give a few important lemmas.

\begin{remark}  \label{feller}
If a random variable $D_n$ converges to $0$ in probability, particularly $ \mathbb{E}\{|D_n|\} \rightarrow 0$, then a second random variable  $U_n$ (on the same probability space) tend to $\Phi$ in distribution if and only if $U_n + D_n \rightarrow \Phi$ in distribution. 
\end{remark}
\bigskip
\begin{remark}\label{moments}
If  distribution function $F_n$ satisfying $\int _{-\infty}^{\infty} x^{k} dF_n (x)\rightarrow \int _{-\infty}^{\infty} x^{k} d\Phi (x)$ as $n\to \infty$, for $k=1, 2, ...$, then $F_{n} (x) \rightarrow \Phi (x)$ for each $x$.\\
\end{remark}
\bigskip
\begin{remark}\label{bound}
If $F_{n}(x) \rightarrow \Phi(x)$ for each x, and if $\int _{-\infty}^{\infty} |x|^{k+\epsilon} dF_{n}(x)$ is bounded in $n$ for some positive $\epsilon$, then, $\int _{-\infty}^{\infty} x^{k} dF_{n}(x) \rightarrow \int _{-\infty}^{\infty} x^{k} d\Phi(x) $.\\
\end{remark}
\bigskip
\begin{remark}\label{normaldis}
 \textit{(A special case of the central limit theorem):} If $X_1,X_2,\dots$ are independent and uniformly bounded random variables with mean $0$ and finite variance $\sigma_i^2$ and if $\sum \sigma_i^2$ diverges then the distribution of $\frac{\sum_{i=1}^n X_i}{\left(\sum_{i=1}^n \sigma_i^2\right)^{1/2}}$ converges to the normal distribution function.
\end{remark}
\bigskip
By recalling ~\cite[Theorem 2.4.]{TenDe} for $m=1$, $d\leq y$ and $y\geq (\log x)^{1+\epsilon}$, we have
\begin{equation}\label{localD}
\Psi(x/d,y)= \frac{\Psi(x,y)}{d^{\alpha}}\left\lbrace1+O\left(\frac{1}{u_y}+\frac{\log d}{\log x}\right)\right\rbrace,
\end{equation}
where $u_y:= u+\frac{\log y}{\log (u+2)}$ and $\alpha=\alpha(x,y)$ denotes the saddle point of the Perron's integral for $\Psi(x,y)$,  which is the solution of the following equation
$$\sum_{p\leq y} \frac{\log p}{p^{\alpha}-1}=\log x.$$
This function will play an important role in this work, so we briefly recall some fundamental facts about this function. By ~\cite[Lemma3.1]{TenDe}, for any $\epsilon >0$, we have the following estimate for $\alpha$
\begin{equation}\label{alphabig} 
\alpha = 1 - \frac{\xi (u)}{\log y}+O\left(\frac{1}{L_\epsilon (y)}+\frac{1}{u(\log y)^2}\right) \qquad \text{if} \qquad y\geq (\log x)^{1+\epsilon},
\end{equation}
where $\xi(u)$ is a unique real non-zero root of the equation $$e^{\xi}= 1+ u\xi,$$
and when $u\ge 3$,  we have 
 \begin{equation}\label{xi}
\xi(u)= \log (u\log u) +O\left(\frac{\log \log u}{\log u}\right).
\end{equation} \\\\
By ~\cite[Lemma 4.1]{TenDe}, we have the following important estimate
\begin{lemma} \label{sumoverp}\textit{(De la Breteche, Tenenbaum)}
For any $x \geq y \ge 2$, uniformly we have
\begin{equation} \label{XXX}
\begin{split}
\sum_{p \le y} \frac{1}{p^{\alpha}} &= 
\log \log y+ \left\lbrace 1+ O\left(\frac{1}{\log y}\right)\right\rbrace \frac{uy}{y+\log x}.
\end{split}
\end{equation}
\end{lemma}
Here we use a particular case of Lemma \ref{sumoverp}. If the range of $y$ is restricted to $\log x < y\le x$, we get
 $$\frac{uy}{y+\log x} = u \left(1+O\left(\frac{\log x}{y}\right)\right),$$
thus,
\begin{equation} \label{sumybig}
\sum_{p \le y} \frac{1}{p^{\alpha}} = \log \log y + u + O\left(\frac{u}{\log y}\right) \qquad y > \log x.
\end{equation}\\\\
For $2\leq t\leq y \leq x$, we define 

$$\omega_t(n) := \# \{p: p|n, p\leq t\}= \sum _{p\leq t} \mathbb{1}_{p|n}.$$
By using the saddle point method, Tenenbaum and de la Breteche in ~\cite{TenDe-Inv} obtained an estimate for the expectation and the variance of $\omega_t(n)$. First, we define 

$$M(t)= M_{x,y}(y):= \sum_{p\leq t} \frac{1}{p^{\alpha}}.$$
We state the following lemma from ~\cite{TenDe-Inv}.
\begin{lemma}\textit{(Tenenbaum, de la Breteche)}\label{M(t)}
we have uniformly for  $2\leq t\leq y\leq x$
\begin{equation}
\mu_{\omega_t}(x,y)=M(t)+O(1).
\end{equation}
\end{lemma}
\bigskip
We now study the expectation of $\omega (n)$, where $n\in S(x,y)$.

\begin{lemma} \label{mean value}
 If $u=o(\log \log y)$, then we have
$$\mu_\omega(x,y) =\log \log y +o(\log \log y).$$
 \end{lemma}
\begin{proof}
Let $t= y$ in Lemma \ref{M(t)}, then we have
$$\mu_\omega (x,y) =\sum_{p\leq y} \frac{1}{p^{\alpha}}+O(1).$$
By using \eqref{sumybig}, we get
$$\mu_\omega(x,y)= \log \log y + u +O(1).$$
Now by letting $u=o(\log \log y)$, we have

$$\mu_\omega(x,y)= \log \log y+ o(\log \log y), $$
and the proof is complete.
\end{proof}
\bigskip
\begin{lemma} \label{sumpt}
If $u=o(\log \log y)$ and $t\leq y^{1/\log u}$, then we have
\begin{equation}\label{Sump<t}
\sum_{p \leq t} \frac{1}{p^{\alpha}} = \log \log t +O(1) 
\end{equation}
\end{lemma}
 \begin{proof}
 We have
$$\sum_{p\leq t} \frac{1}{p^{\alpha}}=\sum_{p\leq t}\frac{1}{pp^{\alpha-1}} = \sum_{p\leq t} \frac{1}{p} \left\lbrace 1+O\left((1-\alpha) \log p\right)\right\rbrace,$$ 
since $(1-\alpha)\log p$ is bounded. By the given estimate for $\alpha$ in \eqref{alphabig} and using Mertens' estimate, we obtain

\begin{equation} 
\begin{split}
\sum_{p\leq t} \frac{1}{p^{\alpha}} &= \sum_{p\leq t}\frac{1}{p} + O\left(\frac{\xi(u)}{\log y} \sum_{p\leq t} \frac{\log p}{p}\right)\\
&= \log \log t + O\left(\frac{\xi(u)}{\log y} \log t\right)
\end{split}
\end{equation}
By applying the estimate of $\xi(u)$ in \eqref{xi}, we get our desired result.
\end{proof}
 \bigskip
Here we will introduce a truncated version of $\omega$ and in the following lemma and corollary we show that the contribution of large prime factors does not affect the expected value of number of prime factors of $n$ and hence the distribution of $\omega(n)$, when $u$ is small enough. We define
\begin{equation}\label{omegaphi}
\omega_{Y} (n): = \sum_{ p \le Y}
 \mathbb{1}_{p|n} (y),
 \end{equation}
 where 
 $$Y:= y^{\frac{1}{\phi(y)}},\quad \text{and} \quad \phi (y) := (\log \log y)^ {\sqrt {\log \log \log y}}.$$
 
  \begin{lemma} \label{truncating}
If $u=o(\log \log y)$, then we have

$$\sum_{p\leq Y} \frac {1}{p^{\alpha}}= \log \log y +O\left((\log \log \log y)^{3/2}\right).$$

\end{lemma}
\begin{proof}
By Lemma \ref{sumpt}, we have
\begin{equation}
\begin{split}
\sum_{p\leq Y} \frac{1}{p^{\alpha}}&= \log \log y- \log \phi(y)+ O(1)\\
& = \log \log y+ (\log \log \log y)^{3/2}+ O(1),
\end{split}
\end{equation}
and we have our desired result.
\end{proof}
Now we define
$$\mu_{\omega_Y} (x,y):= \mathbb{E}\left[\omega_Y(n)\right].$$
In the following lemma we will show $\omega(n)$ can be replaced by $\omega_Y(n)$ in the statement of Theorem \ref{Theorem 1}.
\begin{lemma}
Let $h(n):= \omega(n)- \omega_Y(n)$, then we have
$$\mathbb{P}\left( |h| \leq (\log \log y)^{1/4}\right)= 1-o(1),$$
where $\mathbb{P}$ denotes the probability value.
\end{lemma}
\begin{proof}
We first find an estimate for $\mathbb{E}[h]$, we have
$$\mathbb{E}[h]= \mathbb{E}\left[\omega(n)- \omega_Y(n)\right]= \mu_\omega(x,y)- \mu_{\omega_Y}(x,y).$$
Using Lemma \ref{mean value} and \ref{truncating}, we get
\begin{equation}\label{meanh}
\mathbb{E}[h]\ll (\log \log \log y)^{3/2} \leq (\sqrt{\log \log y}).
\end{equation}
 For the variance of $h$, using \eqref{meanh},  we get
 \begin{equation}\label{varh}
 \begin{split}
 \sigma_h^2(x,y) & :=\mathbb{E}\left[(h-\mathbb{E}[h])^2\right]\\ & =(\mathbb{E}[h])^2 \ll (\log \log \log y)^3.
 \end{split}
  \end{equation}
 Now by Chebyshev's inequality and using \eqref{varh}, we have
 \begin{equation}
 \begin{split}
 \mathbb{E} \left(h \geq (\log \log y)^{1/4}\right) &\leq \mathbb{P} \left( \big| h- \mathbb{E}[h]\big| \geq (\log \log y)^{1/4}\right)\\
 &\leq \frac{\sigma_h^2(x,y)}{(\log \log y)^{1/2}} =o(1),
\end{split}
\end{equation}
and we get our desired result.
\end{proof}
By the above Lemma and recalling Remark \ref{feller}, the estimate in \eqref{result} is equivalent to the following 
\begin{equation}\label{Erd'}
\frac{1}{\Psi (x,y)} \# \{n \in S(x, y):  \frac{\omega_Y(n) - \log \log y}{\sqrt{\log \log y}} \le z \} \rightarrow \Phi(z) \quad (y\to \infty),
\end{equation}
which we prove it in the next section.
\section{Proof of Theorem \ref{Theorem 1}}
We begin this section by setting some random variables $X_p$ on  a probability space and one variable for each prime $p$, which satisfies
\begin{equation}\label{xpdef}
P(X_{p}=1 )=\frac{\Psi (\frac{x}{p}, y)}{\Psi (x, y)}, \quad \text{and} \quad  P(X_{p}=0)= 1- \frac{\Psi (\frac{x}{p}, y)}{\Psi (x, y)}.
\end{equation}
The random variables $X_p$'s are independent. \\\\
Now we define the partial sum $S_Y$ as follows
$$S_Y := \sum_{p\leq Y} X_p,$$
where $Y= y^{1/\phi(y)}.$\\
By the definition of $X_p$'s and the estimate in \eqref{localD} and \eqref{sumybig}, we deduce that $S_Y$ has a mean value and variance of the order $\log \log y$ in the range $u=o(\log \log y)$, this means that $\omega_Y(n)$ and $S_Y$ have roughly the same variance and the same mean value. \\\\

In the following lemma we get an upper bound for the difference of $jth$ moments of $\omega_Y$ and $S_Y$, where $j=1,2,3,\dots.$
\begin{lemma} 
If $u=o(\log \log y)$, then for any positive integer $j$, we have
$$ A_j:=\mathbb{E}_{\substack{n \in S(x, y)}}[\omega_Y(n)^j] - \mathbb{E}[S_Y^j] \ll \frac{(\log \log y)^j}{u(\log \log y)^{\sqrt{\log \log \log y}}}.$$
\end{lemma}
\begin{proof}
By the definition of $\omega_Y$ and $S_Y$, we have
$$\mathbb{E}[\omega_Y^j(n)] =\frac{1}{\Psi(x, y)} \sum_{p_1...p_j \le Y} \sum_{\substack{n \in S(x, y)}} \mathbb{1}_{p_{1}|n}(n)\dots\mathbb{1}_{p_{j}|n}(n),$$
and
$$\mathbb{E}[S_Y^j]=\sum_{p_1...p_j \le Y} \mathbb{E}\left[X_{p_1} \dots X_{p_j} \right].$$
So for the difference of $jth$ moment, we have
\begin{equation}
\begin{split}
A_j &= \sum_{p_1,...,p_j \le y^{\frac{1}{\phi (y)}}}\left(\frac{1}{\Psi(x, y)} \sum_{\substack{n \in S(x, y)}} \mathbb{1}_{p_{1}|n}(n)\dots\mathbb{1}_{p_{j}|n}(n) - \mathbb{E} [X_{p_{1}} ... X_{p_{j}}] \right)\\ &= \sum_{p_1,...,p_j \le y^ {\frac{1}{\phi (y)}}} \left[\frac{\Psi(\frac{x}{p_1 ... p_j}, y)}{\Psi(x, y)} - \prod_{\substack{1 \le i \le j}} \frac{\Psi(\frac{x}{p_i}, y)}{\Psi(x, y)}\right]\\
&=\sum_{p_1,...,p_j \le y^ {\frac{1}{\phi(y)}}}\left[\frac{\Psi(\frac{x}{p_1, ... ,p_j}, y)}{\Psi(x, y)} - \prod_{\substack{1 \le i \le j}} \frac{\Psi(\frac{x}{p_i}, y)}{\Psi(x, y)}\right].
\end{split}
\end{equation}
Without loss of generality we assume that $p_i$'s are distinct, then by using the estimate \eqref{localD}, we have
\begin{equation*}
\begin{split}
A_j &=\sum_{p_1,...,p_j<y^{1/\phi(y)}} \frac{1}{(p_1...p_j)^\alpha} \left\lbrace 1+O\left(\frac{1}{u_y} +\frac{\log p_1...p_j}{\log x}\right)\right\rbrace\\&-\sum_{p_1,...,p_j<y^{1/\phi(y)}} \frac{1}{(p_1...p_j)^\alpha} \prod _{i=1}^j\left\lbrace 1+O\left(\frac{1}{u_y} +\frac{\log p_i}{\log x}\right)\right\rbrace.
\end{split}
\end{equation*}
The main terms in the above subtraction are the same and will be eliminated. Therefore,
\begin{equation}
\begin{split}
A_j&\ll \sum_{p_1,...,p_j<y^{1/\phi(y)}} \frac{1}{(p_1...p_j)^\alpha} \left(\frac{1}{u_y} +\frac{\log p_1...p_j}{\log x}\right)\\
&\ll \sum_{p_1,...,p_j<y^{1/\phi(y)}} \frac{1}{(p_1...p_j)^\alpha} \left(\frac{1}{u_y}+\frac{\log y}{\phi(y)\log x}\right).
\end{split}
\end{equation}
If $u=o(\log \log y)$, then $u_y \geq \frac{\log y}{\log \log \log y}$. So we can ignore the term $\frac{1}{u_y}$. Thus,
$$A_j \ll \sum_{p_1,...,p_j<y^{1/\phi(y)}} \frac{1}{(p_1 \dots p_j)^\alpha} \left(\frac{\log y}{\phi(y)\log x}\right).$$
We now use Lemma \ref{truncating}, and we get the following upper bound for each $A_j$
\begin{equation} \label{Aj}
A_j \ll \frac{(\log \log y)^j}{u(\log \log y)^{\sqrt{\log \log \log y}}}.
\end{equation}
\end{proof}
\bigskip
\begin{proof} [\textbf{Proof of Theorem \ref{Theorem 1}}]
We start our proof by normalizing the random variable $S_Y$.  Define
 $$ S:= \frac{S_Y-\mu_{\omega_Y}(x,y)}{\sqrt{\sigma_{\omega_Y}^{2}(x,y)}}.$$ 
By recalling the central limit theorem, one can say that $S$ has a normal distribution $\Phi(x)$, since $X_p$'s are independent. We set
 $$W:=\frac{\omega_Y (n) - \mu_{\omega_Y}(x,y)}{\sqrt{\sigma_{\omega_Y}^{2}(x,y)}}.$$
 By using the method of moments, we will show that the moments of $W$ are very close to those corresponding sum $S$ and they both converge to the $kth$ moment of normal distribution for every positive integer $k$.\\
By the multinomial theorem, we have
\begin{equation} \label{Delta}
\begin{split}
 \Delta^k &:=\mathbb{E} [(\omega_Y (n) - \mu_{\omega_Y}(x,y))^k] - \mathbb{E}[(S_Y - \mu_{\omega_Y(x,y)})^k] \\&
= \sum_{\substack{j = 1}}^{\substack{k}} \binom{k}{j}\left(-\mu_{\omega_Y}(x,y)\right)^{k-j} \left(\mathbb{E}[\omega_Y (n)^j] - \mathbb{E}[S_Y^j]\right).
\end{split}
\end{equation}
By combining the upper bound in \eqref{Aj} with \eqref{Delta}, we arrive to the following estimate

\begin{equation}
\begin{split}
\Delta^k &\ll \frac{1}{(\log \log y)^{\sqrt{\log \log \log y}}}  \sum_{j=1}^{k} \binom{k}{j} (- \mu_{\omega_Y}(x,y))^{k - j} {(\log \log y)}^{j} \\
&=  \frac{1}{u(\log \log y)^{\sqrt{\log \log \log y}}} \left(\log \log y + \mu_{\omega_Y}(x,y)\right)^{k}.\\
\end{split}
\end{equation}
Now using Lemma \ref{mean value}, we have
\begin{equation}
\begin{split}
\Delta^k &\ll \frac{{(\log \log y)}^{k}}{u(\log \log y)^ {\sqrt{\log \log \log y}}}.
\end{split} 
\end{equation}
Thus, $$\Delta^k \rightarrow 0 \quad \text{as} \quad x,y\rightarrow \infty.$$
We showed that the difference of $kth$ moments goes to $0$ for large values of $y$. By the remark \eqref{moments}, we conclude that two random variables $S$ and $W$ have a same distribution.\\\\\\
By Remark \ref{normaldis}, the random variable $S$ has a normal distribution. It remains to show that the moments of $S$  are very close to those of the normal distribution.\\
 By recalling Remark \ref{bound}, we need to prove that the moment $E[S^{k}]$ are bounded in $n$ when $k$ increases.\\
In fact, we will show that for each $k\in \mathbb{N}$
\begin{equation}\label{bounded}
\sup _{n} \big|\mathbb{E}\left(\frac{(S_{Y} -\mu_{\omega_Y}(x,y))^{k}}{(\sqrt{\sigma_{\omega_Y}^{2} (x,y)} )^{k}}\right)\big| < \infty.
\end{equation}
To complete the proof, we define the random variables $Y_{p}= X_{p} -\frac{\Psi(x/p,y)}{\Psi(x,y)}$, which are independent.\\
 We have

\begin{equation}\label{inner}
\begin{split}
\mathbb{E} \left(\left(S_{Y}- \mu_{\omega_Y} (x,y)\right)^{k}\right)=\sum _{j=1}^{k} \sum ^{'} \frac{k!}{k_1 ! ...k_j !} \sum_{p_1 ...p_j \le y^{\frac{1}{\phi(y)}}} \mathbb{E} [{Y_{p_{1}}^{k_{1}}}]...\mathbb{E} [{Y_{p_{j}}^{k_{j}}}].
\end {split}
\end{equation}
Where $\sum^{'}$ is over $j$-tuple  $(k_1 ,..,k_{j})$, where $k_1,\dots,k_j$ are positive integers, and $k_{1} +...+k_{j} =k$. \\
By the definition of $Y_p's$, we have $\mathbb{E}[Y_{p_{j}}] =0$.\\
To avoid zero terms, we can assume that $k_{i}> 1$ for each $1\leq i\leq j$. Also we have $|Y_{p}| \le 1$. 
Thus, $$\mathbb{E}[Y_{p} ^{k_{i}}] \le \mathbb{E}[Y_{p}^{2}] \quad \forall k_{i}>2.$$
Therefore, the value of inner sum in \eqref{inner} is at most

$$\sum_{p_1 ...p_j \le y^{\frac{1}{\phi(y)}}} \mathbb{E} [{Y_{p_{1}}^{k_{1}}}]...\mathbb{E} [{Y_{p_{j}}^{k_{j}}}] \le \left(\sum_ {p\le {y^{\frac{1}{\phi(y)}}}} \mathbb{E}[Y_{p}^ {2}]\right)^{j} = \sigma^{2j} (x,y).$$
Each $k_{i}$ is strictly greater than $1$, and we have $k_1 +..+k_j = k$, therefore $2j \le k$
 and this implies that
 
 $$\mathbb{E}\left(\frac{\left(S_Y -\mu_{\omega_Y}(x,y)\right)^{k}}{\left(\sqrt{\sigma_{\omega_Y}^{2} (x,y)} \right)^{k}}\right) \le \sum _{j=1}^{k} \sum ^{'} \frac{k!}{k_{1} ! ...k_{j}! }, $$
from which \eqref{bounded} follows. \\We proved all necessary and sufficient conditions such that \eqref{Erd'} and consequently \eqref{result} are true. 
\end{proof}

\bibliographystyle{plain}
\bibliography{references}

\end{document}